\documentclass[12pt,reqno]{amsart}
\usepackage{amsmath}
\usepackage{amsfonts}
\usepackage{amsthm}
\usepackage{amssymb}
\usepackage{hyperref}
\usepackage{fullpage}

\newtheorem{theorem}{Theorem}[section]
\newtheorem{conjecture}[theorem]{Conjecture}
\newtheorem{corollary}[theorem]{Corollary}
\newtheorem{lemma}[theorem]{Lemma}

\theoremstyle{definition}

\numberwithin{equation}{section}
\begin{document}

\title{A density Chinese remainder theorem}
\author{D. Jason Gibson}

\address{Department of Mathematics and Statistics,
Eastern Kentucky University, KY 40475, USA}

\email{jason.gibson@eku.edu}




\date{\today}

\begin{abstract}
Given collections $\mathcal{A}$ and $\mathcal{B}$
of residue classes modulo $m$ and $n$, respectively,
we investigate conditions on $\mathcal{A}$ and $\mathcal{B}$
that ensure that, for at least some $(a,b)\in\mathcal{A}\times\mathcal{B}$,
the system 
\begin{align*}
x&\equiv a\bmod m\\
x&\equiv b\bmod n
\end{align*}
has an integer solution, and we quantify the number of
such admissible pairs $(a,b)$.  The special case
where $\mathcal{A}$ and $\mathcal{B}$ consist
of intervals of residue classes
has application to the Lonely Runner Conjecture.
\end{abstract}

\maketitle

\section{Introduction}

The classical Chinese Remainder Theorem
provides necessary and sufficient conditions
for a system of linear congruence equations
to possess a solution.

\begin{theorem}[Chinese Remainder Theorem]\label{CRT}
Let $m_1,\ldots,m_k$ be $k$ positive integers,
and let $a_1,\ldots,a_k$ be any $k$ integers.
Then the system of congruences
\begin{align*}
x&\equiv a_1\bmod m_1\\
x&\equiv a_2\bmod m_2\\
&\ \ \vdots\\
x&\equiv a_k\bmod m_k
\end{align*}
has a solution if and only if $a_i\equiv a_j\bmod\gcd(m_i,m_j)$
for all pairs of indices $i,j$ with $1\le i<j\le k$.
\end{theorem}

\begin{proof}
See Theorem 7.1 of Hua \cite{Hua}  and
exercises 19 -- 23 in Chapter 2.3 of Niven, Zuckerman,
and Montgomery \cite{NZM}.
\end{proof}

This theorem admits generalization in several
directions.  For a statement of a Chinese
Remainder Theorem in the language of commutative
rings and ideals, see, e.g., Hungerford \cite{hungerford}.
Kleinert
\cite{kleinert} considers a quite general formalism which
yields the usual statement as a special case.
In the sequel, we consider a density Chinese Remainder Theorem
framed in the classical context of systems of two linear
congruence equations.

Specifically, 
given collections $\mathcal{A}$ and $\mathcal{B}$
of residue classes modulo $m$ and $n$, respectively,
we investigate conditions on $\mathcal{A}$ and $\mathcal{B}$
that ensure that, for at least some $(a,b)\in\mathcal{A}\times\mathcal{B}$,
the system 
\begin{align}
\begin{split}
x&\equiv a\bmod m\label{two_eq}\\
x&\equiv b\bmod n
\end{split}
\end{align}
has an integer solution, and we quantify the number of
such admissible pairs $(a,b)$.
For instance, if the collections $\mathcal{A}$
and $\mathcal{B}$ satisfy
$|\mathcal{A}|=|\mathcal{B}|=1$, then the Chinese Remainder Theorem
(Theorem~\ref{CRT})
provides appropriate conditions,
namely,
that the greatest common divisor $\gcd(m,n)$ divides
the difference $a-b$.

For fixed $m$ and $n$,
if the collections $\mathcal{A}$ and $\mathcal{B}$
are not large enough (in a suitable sense),
then it might occur that,
because the $\gcd$
condition doesn't hold,
the system
\eqref{two_eq} admits no solution with $(a,b)\in
\mathcal{A}\times\mathcal{B}$.  Still,
one expects that requiring $\mathcal{A}$
and $\mathcal{B}$ to have large enough density
in comparison to the full set of residue
classes modulo $m$ and $n$, respectively,
will force the $\gcd$ condition to hold,
so that the classical Chinese Remainder Theorem
then yields a solution (or solutions)
to the system of equations.

We make this intuition precise in Theorems
\ref{arbitrary} and \ref{intervals},
establishing a density Chinese Remainder
Theorem for the case of arbitrary
collections of residue classes
and for collections of intervals
of residue classes, respectively.
Results of this form,
particularly generalizations of
Corollary \ref{densitycor}, have application
to the Lonely Runner Conjecture,
which we also briefly discuss.

\section{Intersections of sets of arithmetic progressions}

In this section, we state our main results and some corollaries.
The proofs of Theorems \ref{arbitrary} and \ref{intervals}
appear in the following two sections.

\begin{theorem}[Density CRT]\label{arbitrary}
Let $\mathcal{A}$ and $\mathcal{B}$ be collections of
residue classes modulo $m$ and $n$, respectively.

Let $g=\gcd(m,n)$,
and let $h$ denote the number of solutions to the linear system
\begin{align}
\begin{split}
x&\equiv a\bmod m\label{two_arb_thm}\\
x&\equiv b\bmod n
\end{split}
\end{align}
modulo $mn/g$ as $(a,b)$ ranges over $\mathcal{A}\times\mathcal{B}$.

Write
\begin{align}
\begin{split}
|\mathcal{A}|&=A\frac{m}{g}+r_A,\qquad 0\le r_A<\frac{m}{g},\label{arb_sizes}\\
|\mathcal{B}|&=B\frac{n}{g}+r_B,\qquad 0\le r_B<\frac{n}{g}.\\
\end{split}
\end{align}

Then $h$, the number of solutions to the linear system \eqref{two_arb_thm} modulo $mn/g$,
satisfies
\begin{equation}
h\ge \begin{cases}
0, &\mbox{if } A+B < g-1,\\
r_Ar_B,& \mbox{if }A+B = g-1,\\
(A+B-g)mn/g^2+r_An/g+r_Bm/g,&\mbox{if } A+B > g-1.
\end{cases}
\end{equation}
\end{theorem}

\begin{theorem}[Density CRT for intervals]\label{intervals}
Let $\mathcal{A}$ and $\mathcal{B}$ be intervals of
residue classes modulo $m$ and $n$, respectively.
Let $g=\gcd(m,n)$,
and let $h$ denote the number of solutions to the linear system
\begin{align}
\begin{split}
x&\equiv a\bmod m\label{two_int_thm}\\
x&\equiv b\bmod n
\end{split}
\end{align}
modulo $mn/g$ as $(a,b)$ ranges over $\mathcal{A}\times\mathcal{B}$.

Write
\begin{align}
\begin{split}
|\mathcal{A}|&=Ag+r_A,\qquad 0\le r_A<g,\label{sizes}\\
|\mathcal{B}|&=Bg+r_B,\qquad 0\le r_B<g.\\
\end{split}
\end{align}
Then
$h$, the number of solutions to the linear system \eqref{two_int_thm} modulo $mn/g$,
satisfies
\begin{equation}\label{interval_solution_count}
h\ge ABg
+Ar_B
+Br_A
+\max(0,r_A+r_B-g).
\end{equation}
\end{theorem}

\begin{corollary}[$\gcd(m,n)=1$]

Given the conditions of Theorem \ref{intervals}
with the additional assumption that $\gcd(m,n)=1$,
then the linear system \eqref{two_int_thm}
possesses exactly $|\mathcal{A}||\mathcal{B}|$
solutions.
\end{corollary}

\begin{proof}
This follows immediately from Theorem \ref{intervals}.

Also, more directly, 
note that the condition $\gcd(m,n)=1$
ensures that each choice $(a,b)\in\mathcal{A}\times
\mathcal{B}$ of residue classes leads to a solution modulo $mn$
of the linear system
by way of the classical Chinese Remainder Theorem
(Proposition \ref{CRT}).
\end{proof}

\begin{corollary}[Density statement]\label{densitycor}
Given the conditions of Theorem \ref{intervals},
with the additional assumptions that
\begin{align}
\begin{split}
|\mathcal{A}|&>\frac{1}{3} m, \label{alpha}\\
|\mathcal{B}|&> \frac{1}{3} n,
\end{split}
\end{align}
and that $m$ and $n$ are distinct,
then the linear system \eqref{two_int_thm}
possesses a solution.

Also, the constant $\frac{1}{3}$ can not be taken to be smaller
while keeping the above conclusion of this corollary.
\end{corollary}

\begin{proof}
Following Theorem \ref{intervals},
write
\begin{align}
\begin{split}
|\mathcal{A}|&=Ag+r_A,\qquad 0\le r_A<g,\\
|\mathcal{B}|&=Bg+r_B,\qquad 0\le r_B<g,\\
\end{split}
\end{align}
with $g=\gcd(m,n)$.
To establish the desired claim of this corollary, we must examine
the terms appearing on the right side of \eqref{interval_solution_count}.
If $AB>0$, then the corollary follows immediately, and so it
remains to consider the possibility that $AB=0$.

To that end, suppose that $A=0$ and $B=0$.
(If $B>0$, then the $Br_A$ term in \eqref{interval_solution_count}
is nonzero, and again we are done.)
Then $g>r_A=|\mathcal{A}|>\frac{1}{3}m$
implies that $\frac{m}{g}<3$.  This forces $\frac{m}{g}$ to be either $1$ or $2$.

In the first case, we have $m=g$.  Since $m$ and $n$ are distinct
and $g$ must divide $n$, it must be that $n=mn_0$, where the integer
$n_0$ satisfies $n_0\ge 2$.  Then we have
\begin{equation}\label{B_size}
|\mathcal{B}|>\frac{1}{3}n = \frac{1}{3}mn_0 \ge \frac{2}{3}m=\frac{2}{3}g.
\end{equation}
With $r_A>\frac{1}{3}m=\frac{1}{3}g$ and
$r_B=|\mathcal{B}|>\frac{2}{3}g$,
this establishes that the $r_A+r_B-g$ term of \eqref{interval_solution_count} is nonzero,
and then the linear system posseses a solution by Theorem \ref{intervals}.

In the second case, that $\frac{m}{g}=2$, we have $m=2g$.
Since $m$ and $n$ are distinct and $g$ must divide $n$,
it must be that $n=gn_0$, where the integer $n_0$ satisfies $n_0=1$
or $n_0\ge 3$.  If $n_0=1$, then $m=2g$ and $n=g$, and a solution
exists by the reasoning above.  If $n_0\ge 3$, then
\begin{equation}
|\mathcal{B}|>\frac{1}{3}n=\frac{1}{3}gn_0\ge g
\end{equation}
implies that $B>0$, and, consequently, the $Br_A$ term
in \eqref{interval_solution_count} is nonzero.
Thus, in the case also, a solution again exists by Theorem \ref{intervals}.

Finally,
to see that the constant $\frac{1}{3}$ is optimal, let
$M$ be a large integer, and set $m=3M$, $n=2\cdot 3M$.
Let 
\begin{equation}
\mathcal{A}=\{0,1,\ldots,M-1\}
\end{equation}
and
\begin{equation}
\mathcal{B}=\{M,M+1,\ldots,3M-1\}.
\end{equation}
Then $\gcd(m,n)=\gcd(3M,2\cdot 3M)=3M$,
with $|\mathcal{A}|=M=\frac{1}{3}m$
and $|\mathcal{B}|=2M=\frac{1}{3}n$.
If $(a,b)\in\mathcal{A}\times\mathcal{B}$,
then $a-b$ is nonzero modulo $\gcd(m,n)=3M$,
so that the classical Chinese Remainder Theorem
(Proposition \ref{CRT}) implies that
the linear system has no solution.

\end{proof}

\section{Arbitrary collections}

In this section, we prove Theorem~\ref{arbitrary}.
Given collections $\mathcal{A}$ and $\mathcal{B}$ of
residue classes modulo $m$ and $n$, respectively,
the classical Chinese Remainder Theorem
reduces counting the number of solutions of the linear
system \eqref{two_arb_thm}
to counting the pairs $(a,b)\in\mathcal{A}\times\mathcal{B}$
with $a\equiv b\bmod \gcd(m,n)$.

To that end,
let $g=\gcd(m,n)$, and, for a collection of
residue classes $\mathcal{C}$ and integer $i$ with $1\le i\le g$,
partition the collection $\mathcal{C}$ into sets
\begin{equation}\label{class_definition}
\mathcal{C}_i=\{c\in\mathcal{C}:c\equiv i\bmod g\},
\end{equation}
and define the counting function $f$ by
\begin{equation}\label{f}
f(i,\mathcal{C})=|\mathcal{C}_i|.
\end{equation}
The Chinese
Remainder Theorem then yields that the number of solutions
to the linear system \eqref{two_arb_thm}
is given by the sum
\begin{equation}\label{f_product}
\sum_{i=1}^g f(i,\mathcal{A})f(i,\mathcal{B}).
\end{equation}

Note that we have
\begin{equation}
\begin{split}\label{class_enumeration}
|\mathcal{A}|&=\sum_{i=1}^gf(i,\mathcal{A}),\\
|\mathcal{B}|&=\sum_{i=1}^gf(i,\mathcal{B}).
\end{split}
\end{equation}
Moreover, for $1\le i\le g$, the right-hand summands in
\eqref{class_enumeration} are each bounded above, with
\begin{align}
\begin{split}\label{f_bound}
f(i,\mathcal{A})&\le \frac{m}{g},\\
f(i,\mathcal{B})&\le \frac{n}{g}.
\end{split}
\end{align}

To establish Theorem \ref{arbitrary}, we must find a lower
bound on \eqref{f_product}
under the conditions \eqref{class_enumeration}
and \eqref{f_bound}.
We do so in three steps.  First, we use
an inequality on rearrangements, Lemma \ref{r_ineq},
to reduce to the case where the values $f(i,\mathcal{A})$
are increasing and the values $f(i,\mathcal{B})$ are decreasing.
Next, in this case, Lemma \ref{extremal_dist} yields
an extremal distribution for these values under the constraints of
\eqref{class_enumeration} and \eqref{f_bound}.
Finally, Lemma \ref{sum_lower_bound} provides 
a case analysis for the explicit evaluation
of the sum \eqref{f_product}
in this extremal case.  
Theorem \ref{arbitrary}
follows directly from these three lemmas.

\begin{lemma}[Rearrangement inequality]\label{r_ineq}
For each pair of ordered real sequences
$a_1\le a_2\le \cdots\le a_n$
and
$b_1\le b_2\le\cdots\le b_n$,
and for each permutation $\sigma:[n]\to[n]$,
we have
\begin{equation}
\sum_{k=1}^na_kb_{n-k+1}\le\sum_{k=1}^n
a_kb_{\sigma(k)}\le\sum_{k=1}^n a_kb_k.
\end{equation}
\end{lemma}

\begin{proof}
See Chapter 5 of Steele \cite{masterclass}
and Chapter X of Hardy, Littlewood, and P\'olya
\cite{HLP}.
\end{proof}

\begin{lemma}[Extremal distribution]\label{extremal_dist}
For each pair of non-negative ordered real sequences
$a_1\le a_2\le \cdots\le a_n$
and
$b_1\le b_2\le\cdots\le b_n$
satisfying
$A=\sum_{k=1}^n a_k$,
$B=\sum_{k=1}^n b_k$,
and, for all $k$ with $1\le k\le n$, the bounds $a_k\le q_A$ and
$b_k\le q_B$,
then we have 
\begin{equation}
\sum_{k=1}^n a_k^* b_{n-k+1}^* \le \sum_{k=1}^na_kb_{n-k+1},
\end{equation}
where
\begin{equation}\label{a}
a_k^*=
\begin{cases}
0,&\mbox{if } k< n-\lfloor A/q_A\rfloor,\\
A-q_A\lfloor \frac{A}{q_A}\rfloor, &\mbox{if } k= n-\lfloor A/q_A\rfloor,\\
q_A,&\mbox{if } k> n-\lfloor A/q_A\rfloor,\\
\end{cases}
\end{equation}
and
\begin{equation}\label{b}
b_k^*=
\begin{cases}
0, &\mbox{if } k< n-\lfloor B/q_B\rfloor,\\
B-q_B\lfloor \frac{B}{q_B}\rfloor, &\mbox{if }k= n-\lfloor B/q_B\rfloor,\\
q_B,&\mbox{if } k> n-\lfloor B/q_B\rfloor.\\
\end{cases}
\end{equation} 
\end{lemma}
\begin{proof} Since
$b_1\le b_2\le\cdots\le b_n$, we have
\begin{equation}
\sum_{k=1}^n a_k^* b_{n-k+1} \le \sum_{k=1}^na_kb_{n-k+1}.
\end{equation}
Then, because $a_1^*\le a_2^*\le \cdots \le a_n^*$, we have
\begin{equation}
\sum_{k=1}^n a_k^* b_{n-k+1}^* \le \sum_{k=1}^na_k^*b_{n-k+1},
\end{equation}
and the result follows.
\end{proof}

\begin{lemma}[Extremal sum]\label{sum_lower_bound}
Using the notation of Lemma~\ref{extremal_dist}, if
\begin{equation}
\begin{split}
s&=\lfloor A/q_A\rfloor +\lfloor B/q_B\rfloor+1,\\
r_A&=A-q_A\lfloor A/q_A\rfloor,\\
r_B&=B-q_B\lfloor B/q_B\rfloor,\\
\end{split}
\end{equation}
then we have
\begin{equation}\label{extremal_sum}
\sum_{k=1}^n a_k^* b_{n-k+1}^*=
\begin{cases}
0,&\mbox{if } s< n,\\
r_Ar_B,&\mbox{if } s=n,\\
(s-n-1)q_Aq_B+r_Aq_B+r_Bq_A,
&\mbox{if } s>n.
\end{cases} 
\end{equation}
\end{lemma}

\begin{proof}
First, note that \eqref{b} yields that
\begin{equation}\label{b_reindex}
b_{n-k+1}^*=
\begin{cases}
q_B, &\mbox{if } k< \lfloor B/q_B\rfloor+1,\\
B-q_B\lfloor \frac{B}{q_B}\rfloor, &\mbox{if }k= \lfloor B/q_B\rfloor+1,\\
0,&\mbox{if } k> \lfloor B/q_B\rfloor+1.\\
\end{cases}
\end{equation}
Then, \eqref{a} and \eqref{b_reindex} together
imply that the summands $a_k^*b_{n-k+1}^*$ in \eqref{extremal_sum}
vanish unless $k$ satisfies
\begin{equation}
n-\lfloor A/q_A\rfloor\le k \le \lfloor B/q_B\rfloor +1.
\end{equation}
This allows the sum in \eqref{extremal_sum} to be written as
\begin{equation}
\sum_{k=1}^n a_k^* b_{n-k+1}^*=\sum_{k=n-\lfloor
  A/q_A\rfloor}^{\lfloor B/q_B\rfloor +1}a_k^* b_{n-k+1}^*.
\end{equation}

The analysis of this sum now requires examining three cases:
$s<n$, $s>n$, and $s=n$.
If
\begin{equation}
s=\lfloor A/q_A\rfloor +\lfloor B/q_B\rfloor +1<n,
\end{equation}
then the sum is empty.
If $s>n$, then the sum consists of a first term, last term,
and $s-n-1$ middle terms.  The first and last terms
contribute $r_Aq_B$ and $r_Bq_A$, respectively, while
the remaining terms each have identical value $q_Aq_B$.

Finally, if $s=n$, the sum consists of a single term
of value $r_Ar_B$.
\end{proof}

\section{Collections of intervals}

In this section, we prove Theorem~\ref{intervals}.
As in the previous section, 
given collections $\mathcal{A}$ and $\mathcal{B}$ of
residue classes modulo $m$ and $n$, respectively,
the classical Chinese Remainder Theorem
reduces counting the number of solutions of the linear
system \eqref{two_int_thm}
to counting the pairs $(a,b)\in\mathcal{A}\times\mathcal{B}$
with $a\equiv b\bmod \gcd(m,n)$.  Unlike in the previous section,
the fact that the collections of residue classes
in Theorem~\ref{intervals} consist of intervals
of classes, rather than arbitrary sets of classes,
substantially simplifies the
analysis of this counting problem. 

\begin{proof}[Proof of Theorem~\ref{intervals}]
With $g=\gcd(m,n)$, we have
\begin{align}
\begin{split}
|\mathcal{A}|&=Ag+r_A,\qquad 0\le r_A<g,\\
|\mathcal{B}|&=Bg+r_B,\qquad 0\le r_B<g.\\
\end{split}
\end{align}
It remains to count the pairs $(a,b)\in\mathcal{A}\times\mathcal{B}$
with $a\equiv b\bmod g$.

Since $\mathcal{A}$ consists of intervals
of residue classes, it follows that $\mathcal{A}$ can
be partitioned into $A$ sub-intervals of length $g$,
each sub-interval containing $g$ distinct elements modulo $g$,
along with a smaller sub-interval of length $r_A$,
containing $r_A<g$ distinct elements modulo $g$.
Similarly,
$\mathcal{B}$ can
be partitioned into $B$ sub-intervals of length $g$,
each sub-interval containing $g$ distinct elements modulo $g$,
along with a smaller sub-interval of length $r_B$,
containing $r_B<g$ distinct elements modulo $g$.

Each of the $Ag$ elements from the $A$ sub-intervals
of length $g$ from the collection $\mathcal{A}$
will agree modulo $g$ with exactly one element
in each of the $B$ sub-intervals of length $g$
from the collection $\mathcal{B}$.
These pairings contribute $ABg$ solutions to \eqref{two_int_thm}.

Each of the $r_A$ elements from the smaller sub-interval
of length $r_A$ will agree modulo $g$ with exactly
one element in each of the $B$ sub-intervals of length $g$
from the collection $\mathcal{B}$.
Similarly, each of the $r_B$ elements from the smaller sub-interval
of length $r_B$ will agree modulo $g$ with exactly
one element in each of the $A$ sub-intervals of length $g$
from the collection $\mathcal{A}$. Together, these
pairings contribute $r_AB+r_BA$ solutions to \eqref{two_int_thm}.

Finally, it might be that the two small sub-intervals
of length $r_A<g$ and $r_B<g$ have no elements that agree
modulo $g$.  Still, by the pigeonhole principle,
there must be at least $r_A+r_B-g$ matches modulo $g$.
These pairings contribute $\min(0,r_A+r_B-g)$ solutions
to \eqref{two_int_thm}.

This yields that, in total, \eqref{two_int_thm} possesses
at least
\begin{equation}
ABg+Ar_B+Br_A+\min(0,r_A+r_B-g)
\end{equation}
solutions,
completing the proof of Theorem~\ref{intervals}.

\end{proof}

\section{The Lonely Runner Conjecture}

The Lonely Runner Conjecture, having its origins in view-obstruction
problems and in diophantine approximation, seems to be
due independently to Wills \cite{Wills} and Cusick \cite{Cusick}.
A problem in $n$-dimensional geometry view obstruction
motivated Cusick's statement of the problem, while
Wills viewed the question from the perspective of
Diophantine approximation.

Let $k\ge 2$ be an integer, and let $m_1,\ldots,m_k$ be distinct
positive integers.
For $x\in\mathbb{R}$, let $\|x\|$ denote the
distance from $x$ to the integer nearest,
i.e.,
\begin{equation}
\|x\| = \min\{|x-n|:n\in \mathbb{Z}\}.
\end{equation}
Montgomery includes the Wills version as Problem 44
in the Diophantine Approximation section
of the appendix of unsolved problems in \cite{HLM}.
\begin{conjecture}(Wills)\label{conj_wills} If $1\le m_1<m_2<\ldots<m_k$
then
\begin{equation}
\max_{\alpha\in\mathbb{R}}\min_{1\le k\le K}\Vert m_k\alpha\Vert\ge\frac{1}{k+1}.
\end{equation}
\end{conjecture}
Goddyn, one of the authors of \cite{Goddyn}, gave
the problem a memorable name and interpretation
concerning runners on a circular track.
\begin{conjecture}[Lonely Runner Conjecture]\label{conj_lrc}
Suppose that $k$ runners having positive integer speeds run
laps on a unit-length circular track.  
Then there is a time at which all $k$ runners are at distance at least
$1/(k+1)$ from their common starting point.
\end{conjecture}

The formulation of this problem by Wills and Cusick 
led to a considerable body of work on the Lonely Runner
Conjecture and its various incarnations and applications
in Diophantine approximation (see \cite{Betke}), view-obstruction problems (see
\cite{Cusick} and \cite{Pomerance}),
nowhere zero flows in regular matroids (see \cite{Goddyn}), and
certain graph coloring questions (e.g., \cite{Zhu}, \cite{Liu}).  

At time $t$, runner $i$ has position at distance $\|m_it\|$ from the starting
point.
Call a runner
\emph{distant} at time $t$ if $\|m_it\|\ge \frac{1}{k+1}$.
Given an instance of the lonely runner problem,
we would like to know if there is a time at which
all runners are distant.

To connect this problem to the density Chinese Remainder
Theorem, we proceed by first moving
from its formulation over $\mathbb{R}$
to one over $\mathbb{Q}$, and then
we examine rationals with a suitable
fixed denominator (a function of
the number of runners $k$ and their
speeds $m_1,\ldots,m_k$).

Towards that end, consider a single runner with speed $m$.
We seek to investigate 
the set
\begin{equation}\label{runner set definition}
T(m)=\{t\in \mathbb{R}:\|mt\|\ge\frac{1}{k+1}\},
\end{equation}
the times $t$ at which
this runner is distant.
To understand this set of times, for
an integer $Q\ge 1$, let
\begin{equation}
T_Q(m)=\{t\in T(m): t=a/Q, 0\le a\le Q-1\}.
\end{equation}
Imposing the condition that $(k+1)m$ divides $Q$ ensures
that this set captures the extremal times where equality holds
in \eqref{runner set definition}.  For such $Q$,
the numerators of the times in $T_Q(m)$ correspond
to residue classes modulo $Q$.  These residue classes themselves
are induced by an interval of residue classes modulo $Q/m$.

Hence, for $k$ runners with speeds $m_1,\ldots,m_k$,
we select an appropriate choice of $Q$, and
we seek to find at least one time $t$ that belongs
to each of $T_Q(m_1),\ldots, T_Q(m_k)$.  By the above
remarks, the problem reduces to finding a common solution
to a system of linear congruence equations. 

For example, suppose that $k=2$, with runner
speeds $m$ and $n$. If we set $Q=3mn$,
then the times $T_Q(m)$ are induced by an interval $\mathcal{A}$
of residue classes modulo $3n$, and the times
$T_Q(n)$ are induced by an interval $\mathcal{B}$
of residue classes modulo $3m$.
Specifically,
\begin{equation}
T_Q(m)=\{\frac{a}{3n}:\|\frac{a}{3n}\|\ge\frac{1}{3}, 0\le a\le 3mn-1\},
\end{equation}
and
$\mathcal{A}$ consists of $n+1>\frac{1}{3}\cdot 3n$ residue
classes, the interval of residue classes
\begin{equation}
\{a \bmod 3n: n\le a\le 2n\},
\end{equation}
with analogous statements holding for $T_Q(n)$ and $\mathcal{B}$.
Thus,
Corollary \ref{densitycor} applies, yielding a solution
to the linear system of congruences and hence the
existence of a time belonging to both $T_Q(m)$ and $T_Q(n)$.

\section{Further work}

First, as with the classical Chinese Remainder Theorem,
results for systems containing more than two linear
congruence equations would be both interesting and useful.
In particular, such a generalization in the case
of collections of intervals can be applied to the Lonely
Runner Conjecture.
For instance,
the constant $\frac{1}{3}$ appearing in Corollary \ref{densitycor}
is the analogue of the same constant appearing in the
Lonely Runner Conjecture with $k=2$ runners.
One might hope for such similarities to continue
for larger values of $k$.

Next, other possibilities for the collections $\mathcal{A}$
and $\mathcal{B}$,
for example, random collections (for suitable notions
of random), collections nicely distributed among
residue classes, and collections with other arithmetic
structure might be amenable to analysis.

Finally, regardless of the number of linear congruence
equations or the types of collections involved,
structural information beyond mere existence
and quantity for the admissible classes
and the resulting solutions could be useful
in iterated applications of this, and other, density Chinese
Remainder Theorems.


\end{document}